\documentclass{amsart}

\usepackage{amssymb}
\usepackage{amsthm}
\usepackage{graphicx}
\usepackage{hyperref}
\usepackage{multirow}
\usepackage{array}
\usepackage{comment}
\usepackage[colorinlistoftodos]{todonotes}
\usepackage[english]{babel}
\usepackage{lmodern}

\usepackage[utf8]{inputenc}
\usepackage[T1]{fontenc}
\usepackage{microtype}

\newtheorem{theorem}{Theorem}[section]
\newtheorem{lemma}[theorem]{Lemma}
\newtheorem{corollary}[theorem]{Corollary}

\theoremstyle{definition}
\newtheorem{definition}[theorem]{Definition}
\newtheorem{proposition}[theorem]{Proposition}
\newtheorem{example}[theorem]{Example}

\theoremstyle{remark}
\newtheorem{remark}[theorem]{Remark}

\numberwithin{equation}{section}

\begin{document}

\author{Bartosz Naskręcki}
\title[Mordell-Weil ranks of families of elliptic curves]{Mordell-Weil ranks of families of elliptic curves parametrized by binary quadratic forms}

\address{Faculty of Mathematics and Computer Science, Adam Mickiewicz University\\
Umultowska 87, 61-614 Poznań, Poland\\
and School of Mathematics, University of Bristol, University Walk, Bristol BS8 1TW, UK\\
E-mail: nasqret@gmail.com}
\maketitle

\begin{abstract}
We prove results on the Mordell--Weil rank of elliptic curves $y^2=x(x-\alpha a^2)(x-\beta b^2)$ parametrized by binary quadratic forms $\alpha a^2+\beta b^2=\gamma c^2$. We express our explicit lower bounds over number fields and offer a detailed description of the corresponding Mordell-Weil group structure in the function field case.
\end{abstract}

\section{Introduction}
In the previous paper \cite{Naskrecki_Acta} we have studied a family of elliptic curves
\[y^2=x(x-a^2)(x-b^2),\]
where $a^2+b^2=c^2$ and $a\neq \pm b$ and $ab\neq 0$ over the rationals and over the function fields $\mathbb{Q}(t)$, $\overline{\mathbb{Q}}(t)$. We have computed the non-trivial lower bound for the number of generators of the Mordell-Weil group over $\mathbb{Q}$ as a consequence of more refined result obtained over function fields. In this paper we continue to study families of elliptic curves parametrized in general by binary quadratic forms. This means that we consider the curve of the form
\begin{equation}\label{eq:main_family}
y^2=x(x-\alpha a^2)(x-\beta b^2)
\end{equation}
where $\alpha a^2+\beta b^2=\gamma c^2$ and we consider the solution sets over a fixed number field, as well as, over the function fields of one variable. To each such curve we attach in a suitable sense an elliptic surface fibered over the projective line $\mathbb{P}^{1}$. The arithmetic properties of those surfaces allow us to obtain sharp bounds on the rank of the Mordell--Weil group over the function field $\overline{\mathbb{Q}}(t)$.  Our approach uses as a main tool the Shioda-Tate formula \cite{Shioda_Mordell_Weil} and the explicit intersection pairing defined on elliptic surfaces, which gives a well-defined notion of height of points on elliptic curves over the function field, cf. \S \ref{sub:heights}.

Our main motivation to study this families is to find explicit examples of elliptic curves over the rational function field $\mathbb{Q}(t)$ that have at the same time positive rank and certain fixed torsion subgroup structure. Family of curves \eqref{eq:main_family} appeared already in arithmetic applications in \cite{Ulas_Bremner}. It was also used in \cite{Naskrecki_EDS} to study elliptic divisibility sequences. In \cite{Naskrecki_ALANT} we provide another generalization of family \eqref{eq:main_family}, so we can understand the results of this paper in a larger context. Nonetheless, the results included here form an important step in classification of ranks in family \eqref{eq:main_family}.

\section{Notation}
We will use a common notation for certain objects described in this article. By $E$ we denote an elliptic curve, $K$ is a field of functions over $\mathbb{P}^{1}$, usually $\mathbb{Q}(t)$ or $\overline{\mathbb{Q}}(t)$. By $k$ we denote an algebraically closed field and $\mathcal{E}$ denotes a triple $(S,C,\pi)$ where $\pi:S\rightarrow C$ determines a fibration on $S$ which gives an elliptic surface structure, cf. Definition \ref{definition:ell_surface}. 

\section{Main theorems}
The main theorems are first formulated in the setting of elliptic curves over function fields. Then  by the application of Silverman's specialization theorem we can adopt the results to the arithmetic context of a fixed number field. We also prove as a corollary the result similar to \cite[Thm.1.1]{Naskrecki_Acta} but with improved rank bound by $1$ and the binary quadratic form which is different. We say two polynomials in $\overline{\mathbb{Q}}[t]$ are coprime if they don't have a common root.
\begin{theorem}\label{theorem:Main_theorem_geometric_MW}
Let $f,g\in\overline{\mathbb{Q}}[t]$ be two coprime polynomials such that there exists another polynomial $h\in\overline{\mathbb{Q}}[t]$ that satisfies the relation $f^2+g^2=h^2$. Let us assume that $\deg f=2$ and $\deg g\leq 2$. Let $E$ be an elliptic curve determined by the Weierstrass equation
\begin{equation}\label{equation:Weierstrass_equation_f_g_h}
y^2=x(x-f^2)(x-g^2).
\end{equation}
Then $E(\overline{\mathbb{Q}}(t))\cong \mathbb{Z}^{2}\oplus \mathbb{Z}/2\mathbb{Z}\oplus\mathbb{Z}/4\mathbb{Z}$ and the following points generate the group $E(\overline{\mathbb{Q}}(t))$
\begin{align*}
P_{1}=&(-(1+\sqrt{2})g(g-h),\sqrt{-1}(1+\sqrt{2})g(g-h)(\sqrt{2}g-h)),\\
P_{2}=&((f-h)(g-h),(f+g)(f-h)(g-h)),\\
T_{1}=&(g^2,0),\\
T_{2}=&(fg, \sqrt{-1}f(f - g)g).
\end{align*}
\end{theorem}

Let us fix a number field $F$. Take $\alpha,\beta,\gamma\in F$ nonzero elements such that $\beta^2+4\alpha\gamma\neq 0$. Define a quadratic polynomial $q_{\alpha,\beta,\gamma}(a,b,c)=\alpha a^2+\beta b^2-\gamma c^2\in F[a,b,c]$. We say that the quadric $q_{\alpha,\beta,\gamma}(a,b,c)=0$ is \textit{parametrizable} if there exists a tuple $(a_{0},b_{0},c_{0})\in F^3$ not equal to $(0,0,0)$ such that the equality $q_{\alpha,\beta,\gamma}(a_{0},\beta_{0},\gamma_{0})=0$ holds. It follows that the quadric $q_{\alpha,\beta,\gamma}(a,b,c)=0$ is parametrizable if and only if it has infinitely many solutions. This holds if and only if there exist three polynomials $f,g,h\in F[t]$ such that for any triple $(A,B,C)\in F^3$ that satisfies $q_{\alpha,\beta,\gamma}(A,B,C)=0$ we can find a number $t\in F$ such that
\[\frac{A}{C}=\frac{f(t)}{h(t)},\quad \frac{B}{C}=\frac{g(t)}{h(t)}.\]
The equation $q_{\alpha,\beta,\gamma}(a,b,c)=0$ defines a projective curve, a conic $C$ over $F$ in $\mathbb{P}^{2}_{F}$. The quadric $q_{\alpha,\beta,\gamma}(a,b,c)=0$ is parametrizable if and only if $C(F)\neq \emptyset$. There is a standard procedure which gives an isomorphism $C\cong \mathbb{P}^{1}_{F}$. Let $P=(a_{0},b_{0},c_{0})$ be a closed point in $C(F)$. We consider a pencil $\mathcal{L}$ of lines in $\mathbb{P}^{2}$ that pass through $P$. Each line $\ell\in\mathcal{L}$ defined over $F$ which is not tangent to $C$ intersects $C(F)$ in two distinct points $\{P,P_{\ell}\}=C(F)\cap \ell(F)$. Each point $P_{\ell}$ can be described in the homogeneous coordinates as $[f(t_{\ell}):g(t_{\ell}):h(t_{\ell})]$ where $f,g,h\in F[t]$ are polynomials of degree at most $2$ and $t_{\ell}\in F$. The polynomials $f,g,h$ depend only on the choice of the pencil $\mathcal{L}$. Without loss of generality we can assume that $2=\deg h=\deg f\geq \deg g$, permuting coordinates $a$ and $b$ if necessary.

\begin{theorem}\label{theorem:main_theorem_ranks_over_number_fields}
Let $F$ be a number field and let $\alpha,\beta,\gamma\in F^{\times}$. Assume that the quadric $q_{\alpha,\beta,\gamma}(a,b,c)=0$ is parametrizable. There exist an infinite set of triples $(a,b,c)\in F^3$ that satisfy $q_{\alpha,\beta,\gamma}(a,b,c)=0$ and the Weierstrass equation
\[y^2=x(x-\alpha a^2)(x-\beta b^2)\]
determines an elliptic curve $E_{\alpha,\beta,\gamma}$ defined over $F$. The rank of $E_{\alpha,\beta,\gamma}(F)$ equals at least
\begin{itemize}
\item[(i)] $1$,\quad when $-2\gamma\in (F^{\times})^{2}$,
\item[(ii)] $1$,\quad when $\alpha\beta\gamma\in (F^{\times})^{2}$,
\item[(iii)] $2$,\quad when $\alpha\beta\gamma\in (F^{\times})^{2}$ and $-2\gamma\in (F^{\times})^{2}$.
\end{itemize}
\end{theorem} 
For certain choices of $\alpha,\beta$ and $\gamma$ the rank bound from Theorem \ref{theorem:main_theorem_ranks_over_number_fields} can be improved. Let $f_{1}=t^{2}+2^{5}$, $g_{1}=-2^{4}t$ and $h_{1}=-(t^{2}-2^{5})$ be polynomials in $\mathbb{Q}[t]$. Define a set
\begin{equation}\label{equation:specialized_rank_3_set}
S=\left\{\left(f_{1}(\frac{-2^4 t}{-10+t^2}),g_{1}(\frac{-2^4 t}{-10+t^2}),h_{1}(\frac{-2^4 t}{-10+t^2})\right):t\in\mathbb{Q}^{\times}\right\}
\end{equation}
of triples of rational numbers. Observe that if $(a,b,c)$ belongs to $S$, then $-2a^2+b^2=-2c^2$. The following theorem holds.
\begin{corollary}\label{corollary:Best_rank_result}
There exists a finite subset $S_{0}\subset S$ such that for all $(a,b,c)\in S\setminus S_{0}$ the curve
\begin{equation}
E_{a,b,c}: y^2 = x(x+2a^2)(x-b^2)
\end{equation}
is elliptic and the rank of the group $E_{a,b,c}(\mathbb{Q})$ is at least $3$.
\end{corollary}
\begin{remark}\label{remark:explicit_rank_3_subgroup}
The corollary above gives and example on how we can improve the rank result stated in Theorem \ref{theorem:main_theorem_ranks_over_number_fields}. In particular Theorem \ref{theorem:main_theorem_ranks_over_number_fields} implies that $\textrm{rank } E_{a,b,c}(\mathbb{Q})\geq 2$ and explicitely there exist a pair of linearly independent points, namely
\begin{align*}
\widetilde{Q}_{1}=&(-b^2,-2b^2c),\\
\widetilde{Q}_{2}=&(-2c^2,-2abc).
\end{align*}
In Corollary \ref{corollary:Best_rank_result} we raise the rank by one at the cost of making the set of admissible triples $(a,b,c)$ smaller but still infinite. In this particular situation we can check that if $(a,b,c)\in S$ then $2(-32 + a)(64 a + b^2)$ is a square in $\mathbb{Q}$, hence we can find yet another linearly independent point in $E_{a,b,c}(\mathbb{Q})$
\[\widetilde{Q}_{3}=\left(-2^6 a,2^3 a\sqrt{2(-32 + a)(64 a + b^2)}\right).\]
\end{remark}

\begin{remark}\label{remark:rank_3_curve_E_4}
As will be explained later, the result obtained in Corollary \ref{corollary:Best_rank_result} follows from the fact that to the described family we can attach an elliptic surface such that the generic fiber treated as a curve over $\mathbb{Q}(t)$ has Mordell-Weil rank equal to $3$. 
\end{remark}

\section{Elliptic surfaces vs. families}
We will use frequently the notion of elliptic surfaces in what follows, so we recall it in the context that is necessary in this article.
\begin{definition}\label{definition:ell_surface}
Let $k$ be an algebraically closed field. Let $C$ be a smooth projective curve over $k$ and $S$ be a smooth projective surface over $k$. We call a triple $(S,C,\pi)$ an elliptic surface when $\pi:S\rightarrow C$ is a surjective morphism such that
\begin{itemize}
\item there exists a non-empty set $B\subset C(k)$ such that for any $v\in C(k)\setminus B$ the fiber $\pi^{-1}(v)$ is a curve of genus $1$,
\item there exists a section $O: C\rightarrow S$ of the morphism $\pi$,
\item no fiber $\pi^{-1}(v)$ for $v\in C(k)$ contains $(-1)$-curves.
\end{itemize}
\end{definition}

To any elliptic curve over $F(t)$ we can attach the corresponding elliptic surface fibered over $\mathbb{P}^{1}_{F}$. We call it a Kodaira-N\'{e}ron model of $E$ over $F(t)$.

We associate with an element $a\in k$ the function $v_{a}:k(t)\rightarrow \mathbb{Z}\cup\{\infty\}$ which assigns to a rational function $g\in k(t)$ its order of vanishing $v_{a}(g)$ at point $a$. Our convention is that $v_{a}(0)=\infty$. Function $v_{a}$ defines a discrete valuation on the field $k((t-a))$ of Laurent polynomials of variable $t-a$. We should emphasize the role of $k$, but in our applications it will always be a fixed algebraic closure $\overline{\mathbb{Q}}$ of the field of rational numbers $\mathbb{Q}$.

When $E$ is a Weierstrass model of an elliptic curve over $F(t)$ for $F$ a number field or $\overline{\mathbb{Q}}$, we say that the equation $E$ is $v_{a}$-minimal if it is defined over $F[t]$ and is $v_{a}$-minimal in the usual sense as a model of elliptic curve over the local field $\overline{F}((t-a))$.

For an elliptic surface $(S,\mathbb{P}^{1}_{k},\pi)$ the preimage with respect to $\pi$ of the generic point is an elliptic curve $E$ over the function field $k(\mathbb{P}^{1}_{k})$. There is a small ambiguity of the choice of the local parameter that generates the function field $k(\mathbb{P}^{1}_{k})$ and which also determines a corresponding Weierstrass equation in local coordinates. For $t\in k(\mathbb{P}^{1}_{k})$ such that $t([X:Y])=X/Y$ we write a model of $E$
\[E_{1}:y^2+a_{1}xy+a_{3}y=x^3+a_{2}x^2+a_{4}x+a_{6},\quad a_{1},a_{2},a_{3},a_{4},a_{6}\in k(t).\]
For the function $s\in k(\mathbb{P}^{1}_{k})$ such that $s([X:Y])=Y/X$ we define a model
\[E_{2}:(y')^2+a'_{1}x'y'+a'_{3}y'=x'^3+a'_{2}x'^2+a'_{4}x'+a'_{6},\quad a'_{1},a'_{2},a'_{3},a'_{4},a'_{6}\in k(s).\]
When the model $E_{1}$ is $v_{a}$-minimal for some $a\neq 0$, it does not necessarily implies that the model $E_{2}$ is also $v_{a}$-minimal. In order to achieve a model that is optimal for computations, we first minimize it with a local parameter $t-a$ at all places $v_{a}$, where $a\in k$. This is always possible since $k[t]$ is a principal ideal domain, cf. \cite[VIII,\S 8]{Silverman_arithmetic}. Then, we replace in our new model $E_{1}'$ minimal at all $a$, the variable $t$ by $1/s$ and we perform a change of coordinates $(x,y)\mapsto (x/s^{2n},y/s^{3n})$ and we choose the least integer $n$ that satisfies for all $i$ the condition $\deg_{t}(a_{i}(t))\leq n i$. Our new model $E_{2}'$ will have the coefficients $a_{i}'(s)=s^{ni}a_{i}(1/s)\in k[s]$. Now, we can check if our model $E_{2}'$ is minimal at $s=0$ (we say $v_{\infty}$-minimal or minimal at $\infty$).

If the model $E_{1}'$ is minimal at all $a\in k$ and at $\infty$ we say that it is \textit{globally minimal}.

Application of the Chinese remainder theorem like in \cite[VII.8]{Silverman_arithmetic} allows us to assume that the globally minimal model can be obtained from the original model by a change of coordinates
\[x\mapsto u^2 x'+r\quad y\mapsto u^3 y'+u^2 s x'+w\]
between two Weierstrass forms over $k(t)$ where $u,s,r,w$ all lie in $k[t]$. We say such a change of coordinates is \textit{admissible} if $u$ is a nonzero constant in $k$. Every globally minimal model is unique up to an admissible change of coordinates. We assume $a_{i}$ are the coefficients of the original Weierstrass form and $a_{i}'$ are the coefficients of the form after the transformation. By a direct computation we get the following identities

\begin{align*}
ua_{1}'&=a_{1}+2s\\
u^2 a_{2}'&=a_{2}-sa_{1}+3r-s^2\\
u^3a_{3}'&=a_{3}+ra_{1}+2w\\
u^4a_{4}'&=a_{4}-sa_{3}+2ra_{2}-(w+rs)a_{1}+3r^2-2sw\\
u^6a_{6}'&=a_{6}+ra_{4}+r^2a_{2}+r^3-wa_{3}-t^2-rwa_{1}
\end{align*}
There is a useful criterion which makes it easy to check when the model is globally minimal.

\begin{theorem}[Globally minimal Weierstrass model]\label{theorem:globally_minimal_condition}
Let $E$ be an elliptic curve over $k(t)$. The Weierstrass model of curve $E$ with coefficients $a_{i}(t)\in k(t)$ is globally minimal if and only if there exists an $n\in\mathbb{N}$ such that the following conditions hold
\begin{itemize}
\item[(i)] for any $i$ we have $a_{i}\in k[t]$,
\item[(ii)] for any $i$ the inequality $\deg a_{i}(t)\leq n i$ holds,
\item[(iii)] there exists an $i$ such that $\deg a_{i}'(t) \geq (n-1)i$, where coefficients $a_{i}'$ come from any admissible change of coordinates,
\item[(iv)] for any $a\in k$ there exists an $i$ such that $v_{a}(a_{i}')<i$, where coefficients $a_{i}'$ come from any admissible change of coordinates,
\end{itemize}
\end{theorem}
\begin{proof}
The proof is based on \cite[\S 8.2]{Shioda_Schutt}.

($\Rightarrow$) We assume that the model of $E$ is globally minimal. By the very definition $a_{i}\in k[t]$ and for every $a\in k$ there exists an $i$ such that $v_{a}(a_{i})$ -- the order of vanishing at $a$ of the function $a_{i}$ satisfies $v_{a}(a_{i})<i$. If not, then for every $i$ and some $a\in k$ we would have $v_{a}(a_{i})\geq i$ and the change of coordinates $(x,y)\mapsto (x/(t-a)^2,y/(t-a)^3)$ would decrease the valuation $v_{a}$ of $a_{i}$ but will not destroy the property $a_{i}\in k[t]$. The same property will hold for any model obtained by an admissible change of coordinates, cf. \cite[VII, Prop. 1.3(b)]{Silverman_arithmetic} For $a=\infty$ the minimality means that there exists a natural number $n$ such that $a_{i}'(s)=s^{ni}a_{i}(1/s)\in k[s]$, which is equivalent to $\deg_{t}(a_{i})\leq ni$. Moreover, from the minimality at $\infty$ we deduce that there is an $i$ such that $v_{s}(a_{i}'(s))\leq i$, equivalently $\deg_{t}(a_{i}(t))\geq (n-1)i$. Again this will hold for any admissible change of coordinates. This finishes the proof of the implication.

($\Leftarrow$) We will prove the implication

\textit{(*) If the model of $E$ is not globally minimal, then for every $n\in\mathbb{N}$ the alternative of negations of conditions (i),(ii),(iii), (iv) holds.}

Let us assume that the model of $E$ is not globally minimal. If for some $i$ we have $a_{i}\notin k[t]$, then the condition (i) cannot hold, so implication (*) is true. We assume from now on that (i) holds.

If for an $a\in k$ the model of $E$ is not $v_{a}$-minimal, then condition (iv) does not hold and (*) is true. So we assume now also that (iv) holds.

For sufficiently small $n\in\mathbb{N}$ there is an $i$ such that $\deg_{t} (a_{i}(t))>ni$, then (ii) can't hold and (*) is true.

So now we assume that $n$ is sufficiently big. If the model of $E$ is not globally minimal and (iv) holds, then it can't be minimal at $\infty$. But from condition (ii) it follows that the coefficients $s^{ni}a_{i}(1/s)$ all lie in  $k[s]$. So for all $i$ we have the inequality $v_{s}(a_{i}')>i$, where $a_{i}'$ come from an admissible change of coordinates of the model with coefficients $s^{ni}a_{i}(1/s)$. This is equivalent to $ni-\deg_{t}(a_{i}(t))>i$ where $a_{i}$ might come from an admissible change of coordinates. So the condition (iii) does not hold, which is a contradiction, hence (*) holds.
\end{proof}

\begin{remark}
Given the globally minimal Weierstrass equation over $\overline{\mathbb{Q}}(t)$, a point $(x(t),y(t))$ will transform after the change of coordinates $t\mapsto 1/s$ into 
\[(x(1/s)/s^{2n},y(1/s)/s^{3n})\]
where $n$ is the least integer $n$ determined by Theorem \ref{theorem:globally_minimal_condition}.
\end{remark}

\section{Proofs}

\begin{lemma}\label{lemma:globally_minimal_model_f_g}
Let $f,g\in\overline{\mathbb{Q}}[t]$ be two coprime polynomials. The equation
\begin{equation}\label{eq:general_equation_f_g}
y^2=x(x-f^2)(x-g^2)
\end{equation}
is a globally minimal Weierstrass model.
\end{lemma}
\begin{proof}
Without loss of generality we can assume that $\deg g\leq \deg f$. For $i\in \{1,3,6\}$ we have $a_{i}=0$ and $a_{2}=-f^2-g^2$, $a_{4}=f^2 g^2$. Since all $a_{i}(t)$ are polynomials, condition (i) from Theorem \ref{theorem:globally_minimal_condition} is satisfied. Let $n=\deg f$. It follows that $\deg a_{2}\leq 2\max\{\deg f,\deg g\}=2n$ and $\deg a_{4}=2(\deg f+\deg g)\leq 4\deg f=4n$, thus condition (ii) holds. 

\noindent
Let us prove that condition (iv) holds true. Our local field is $K=\overline{F}((t-a))$ and $v=v_{a}$. We denote by $R$ the ring of integers of $K$ with respect to $v$. We have to analyze the valuations $v(a_{i}')$ for all models that come from admissible change of coordinates of the original model.
\begin{align}
v(a_{1}')&=v(2s)=v(s)\\
v(a_{2}')&=v(a_{2}+3r-s^2)\\
v(a_{3}')&=v(2w)=v(w)\\
v(a_{4}')&=v(a_{4}+2ra_{2}+3r^2-2sw)\label{eq:a4}\\
v(a_{6}')&=v(ra_{4}+r^2a_{2}+r^3-w^2)\label{eq:a6}
\end{align}
Assume that $v(a_{i}')\geq i$ for all $i$. Let $\varpi$ be a certain uniformizer for $v$, namely $v(\varpi)=1$. Then it readily follows that $s=s_{0}\cdot \varpi$ and $w=w_{0}\cdot \varpi^3$ for certain $s_{0},w_{0}\in k[t]$. Next we have $a_{2}+3r-s^2=a_{2,0}\cdot \varpi^2$ with $a_{2,0}\in k[t]$. This implies the equality
\begin{equation}\label{a2_val}
a_{2}+3r=a_{2,1}\varpi^2
\end{equation}
with $a_{2,1}\in R$. From \eqref{eq:a4} it follows that
\begin{equation}\label{eq:a4_val}
a_{4}+2ra_{2}+3r^2=a_{4,0}\varpi^{4}
\end{equation}
with $a_{4,0}\in R$. Then $a_{4}+2ra_{2}+3r^2=a_{4}+ra_{2}+r(a_{2}+3r)=a_{4}+ra_{2}+r(a_{2,1}\varpi^2)$ and combining this with the previous equation implies $a_{4}+ra_{2}=a_{4,1}\varpi^{2}$ where $a_{4,1}\in R$. Equation \eqref{eq:a6} implies that $ra_{4}+r^2a_{2}+r^3=a_{6,0}\varpi^6$ and then
\[r(a_{4,1}\varpi^2+r^2)=a_{6,0}\varpi^6.\]
It follows that $v(r)\geq 1$. Then equation \eqref{eq:a4_val} implies that $v(a_{4})\geq 1$ and from \eqref{a2_val} we get that $v(a_{2})\geq 1$. Explicitly this implies that polynomials $f^2+g^2$ and $f^2 g^2$ have a common root, which contradicts the assumption $f,g$ being coprime.

\medskip\noindent
We deal now with condition (iii). We assume from the beginning that $\deg g\leq \deg f$. Put $n=\deg f$. Let us assume now that for any fixed admissible change of coordinates condition (iii) is not satisfied, namely
\begin{align}
\deg(s)&<(n-1)\label{al:a1_deg}\\
\deg(a_{2}+3r-s^2)&<2(n-1)\label{al:a2_deg}\\
\deg(w)&<3(n-1)\label{al:a3_deg}\\
\deg(a_{4}+2ra_{2}+3r^{2}-2sw)&<4(n-1)\label{al:a4_deg}\\\
\deg(ra_{4}+r^{2}a_{2}+r^{3}-w^{2})&<6(n-1)\label{al:a6_deg}
\end{align}
From equation \eqref{eq:general_equation_f_g} we have $a_{2}=-f^{2}-g^{2}$ and $a_{4}=f^{2}g^{2}$. Let us proceed first with the case $\deg g <\deg f$.
From the assumptions we get $\deg(a_{4})<4n$ and $\deg(a_{2})=2n$. From \eqref{al:a2_deg} it follows that $\deg r =2n$. Hence by \eqref{al:a4_deg} we get $\deg(2ra_{2}+3r^{2})<4n$. Also we know that $\deg(a_{2}+3r)<2n$ by \eqref{al:a2_deg}, so $\deg(ra_{2}+3r^{2})<4n$ by additive property of the degree function. This implies the inequality $\deg(ra_{2})<4n$, which gives a contradiction.

\medskip\noindent
We assume for the next part that $\deg f=\deg g$. It follows that the equality $\deg(a_{4})=4n$ holds, but for the other coefficient we have only the inequality $\deg(a_{2})\leq 2n$. Inequalities \eqref{al:a1_deg},\eqref{al:a3_deg} and \eqref{al:a4_deg} combined with $\deg(a_{4})=4n$ imply equality $\deg(2ra_{2}+3r^{2})=4n$. Let us denote the leading coefficient of $a_{2}$ by $a_{2}^{0}$, of $r$ by $r^{0}$ and of $a_{4}$ by $a_{4}^{0}$. From the previous equality we deduce that
\begin{equation}\label{eq:first_leading}
a_{4}^{0}+2r^{0}a_{2}^{0}+3(r^{0})^2=0.
\end{equation}
We also have $\deg(r)+\deg(2a_{2}+3r)=4n$. If we assume $\deg(r)<2n$, then $\deg(2ra_{2}+3r^{2})<4n$, a contradiction. So let $\deg(r)\geq 2n$. Inequality \eqref{al:a2_deg} and $\deg(a_{2})\leq 2n$ imply now that $\deg(r)=2n$ and $\deg(a_{2})=2n$, so
\begin{equation}\label{eq:second_leading}
a_{2}^{0}+3r^{0}=0.
\end{equation}
From \eqref{al:a3_deg} and \eqref{al:a6_deg} we deduce $\deg(ra_{4}+r^{2}a_{2}+r^{3})<6(n-1)$. But each term of the polynomial on the left-hand side has degree $6n$, so we obtain the final piece
\begin{equation}\label{eq:third_leading}
r^{0}a_{4}^{0}+(r^{0})^{2}a_{2}^{0}+(r^{0})^{3}=0.
\end{equation}
Now we use \eqref{eq:first_leading} and \eqref{eq:second_leading} to get $a_{4}^{0}=3(r^{0})^{2}$. We divide both sides of \eqref{eq:third_leading} by $r^{0}$ and substitute $a_{2}^{0}$ from \eqref{eq:first_leading}. That implies equality $a_{4}^{0}=2(r^{0})^2$. Hence $2(r^{0})^2=3(r^{0})^2$, a contradiction. So we have proved that condition (iii) holds for all admissible changes of coordinates.

\end{proof}

\noindent
By the results of Oguiso \cite[Theorem 1]{Oguiso_c2} and Shioda \cite[Theorem 2.8]{Shioda_Mordell_Weil} if $E$ is an elliptic curve over $F(t)$ with Kodaira-N\'{e}ron model $(S,\mathbb{P}^{1}_{\overline{F}},\pi)$, the Euler characteristic $\chi(S)=\chi(S,\mathcal{O}_{S})$ is equal to the least number $n$ specified in Theorem \ref{theorem:globally_minimal_condition}. Hence by the proof of Lemma \ref{lemma:globally_minimal_model_f_g} we have $\chi(S)=\deg f$ for the Kodaira-N\'{e}ron model of the curve given by (\ref{eq:general_equation_f_g}).

\begin{lemma}\label{lemma:bad_fibres_types_f_g_curve}
Let $f,g$ be as in previous lemma. The elliptic curve determined by the equation (\ref{eq:general_equation_f_g}) corresponds to an elliptic surface $\mathcal{E}=(S,\mathbb{P}^{1}_{\overline{\mathbb{Q}}},\pi)$ such that all bad fibers are of Kodaira type $I_{n}$ for a suitable value of $n$. More precisely, the discriminant $\Delta$ of equation (\ref{eq:general_equation_f_g}) is
\[16 f^4 g^4 \left(f^2-g^2\right)^2.\]
Moreover
\begin{itemize}
\item[(i)] if $a$ is a root of $f$ or $g$ of multiplicity $e$, then the fiber $\pi^{-1}(a)$ is of type $I_{4e}$,
\item[(ii)] if $a$ is a root of $f^2-g^2$ of multiplicity $e$, then the fiber $\pi^{-1}(a)$ is of type $I_{2e}$,
\item[(iii)] if $a=\infty$ and $\deg f\geq \deg g$, then the fiber $\pi^{-1}(a)$ is of type $I_{n}$ where $n=8\deg f-4\deg g-2\deg(f^2-g^2)$.
\end{itemize}
\end{lemma}
\begin{proof}
By Lemma \ref{lemma:globally_minimal_model_f_g} the model of $E$ given by (\ref{eq:general_equation_f_g}) is globally minimal. We apply Tate's algorithm \cite{Tate_algorithm_article} to $\mathcal{E}$. Conditions (i) and (ii) of the theorem follow. We assume that $\deg f\geq \deg g$, hence $\chi(S)=\deg f$ and $v_{\infty} (\Delta) =12\deg f-\deg(\Delta)$, $\Delta$ being the discriminant of (\ref{eq:general_equation_f_g}). The change of coordinates $(x,y)\mapsto (x/s^{2\chi(S)},y/s^{3\chi(S)})$ exhibits the minimal model with respect to $s$ (at $\infty$). The reduction type at $\infty$ is therefore $I_{n}$, where $n=v_{\infty}(\Delta)$, again by Tate's algorithm. This completes the proof of (iii).
\end{proof}

\subsection{Torsion subgroup}
In this paragraph we want to compute the group structure of torsion points on curves (\ref{eq:general_equation_f_g}). This is used then in the next sections to establish the structure of the full Mordell-Weil group through the theory of lattices.

\begin{lemma}\label{lemma:torsion_subgroup_general}
Let $f,g\in\overline{\mathbb{Q}}[t]$ be two coprime polynomials. On the curve $E$ with the Weierstrass equation (\ref{eq:general_equation_f_g}) there are two points
\begin{align*}
T_{1}=&(g^2,0),\\
T_{2}=&(fg, \sqrt{-1}f(f - g)g)
\end{align*}
which span a subgroup in $E(\overline{\mathbb{Q}}(t))$ isomorphic to $\mathbb{Z}/2\mathbb{Z}\oplus\mathbb{Z}/4\mathbb{Z}$.  
\end{lemma}
\begin{proof}
Let $P=(x,y)\in E(\overline{\mathbb{Q}}(t))$ be a fixed $\overline{\mathbb{Q}}(t)$-rational point. If $P$ is of order $2$ then $P=-P$, hence $y=0$. This implies that $P\in\{(0,0),(g^2,0),(f^2,0)\}$. If $P$ is not of order two, then by the duplication formula we get the $x$-coordinate of the point $2P$
\begin{equation}\label{eq:point_duplication}
x(2P)=\frac{(x-f g)^2 (f g+x)^2}{4 x \left(x-f^2\right) \left(x-g^2\right)}.
\end{equation}
By the formula (\ref{eq:point_duplication}) we get that $x(2T_{2})=0\neq g^2$. We also get that $y(2T_{2})=0$, so $2T_{2}$ is a point of order $2$ different from $T_{1}$. The statement of the lemma follows.
\end{proof}

\begin{corollary}\label{corollary:torsion_structure_f_g_curves_with_assumptions}
Let $f,g\in\overline{\mathbb{Q}}[t]$ be two coprime polynomials. Assume that $\deg f=2$ and $\deg g\leq \deg f$. Moreover, let $f^2-g^2$ be separable. For the elliptic curve $E$ over $\overline{\mathbb{Q}}(t)$ determined by the equation (\ref{eq:general_equation_f_g}) the following holds
\[E(\overline{\mathbb{Q}}(t))_{\textrm{tors}}\cong \mathbb{Z}/2\mathbb{Z}\oplus\mathbb{Z}/4\mathbb{Z}.\]
\end{corollary}
\begin{remark}
Points $T_{1},T_{2}$ from Lemma \ref{lemma:torsion_subgroup_general} span the group $E(\overline{\mathbb{Q}}(t))_{\textrm{tors}}$.
\end{remark}
\begin{proof}
Let $\mathcal{E}=(S,\mathbb{P}^{1}_{\overline{\mathbb{Q}}},\pi)$ be the Kodaira-N\'{e}ron model of the curve $E$. By Lemma \ref{lemma:bad_fibres_types_f_g_curve} we know that the elliptic surface $\mathcal{E}$ has bad fibers of types $I_{2}$ and $I_{4}$. Let $B$ denote the subset of $\mathbb{P}^{1}_{\overline{\mathbb{Q}}}(\overline{\mathbb{Q}})$ of points such that for $v\in B$ the fiber $F_{v}=\pi^{-1}(v)$ is not smooth. Let $G(F_{v})$ denote the group of simple components of $F_{v}$ (component group of the N\'{e}ron model of $E$ with respect to valuation at $v$). Then by \cite[Corollary IV.9.2]{Silverman_book} and \cite[Corollary 7.2]{Shioda_Schutt} there exists an injective homomorphism
\begin{equation}\label{equation:injective_torsion_homomorphism}
E(\overline{\mathbb{Q}}(t))_{\textrm{tors}}\hookrightarrow \prod_{v\in B}G(F_{v}).
\end{equation}
For the fibers of type $I_{n}$ the group $G(I_{n})$ equals $\mathbb{Z}/n\mathbb{Z}$. This implies that in our case the group $E(\overline{\mathbb{Q}}(t))_{\textrm{tors}}$ can contain only points of orders dividing $8$. The proof of Lemma \ref{lemma:bad_fibres_types_f_g_curve} implies that $\chi(S)=\deg f=2$, hence $S$ the geometric genus $p_{g}$ of surface $S$ equals $1$. By \cite[Thm. 2.2]{Cox_rational_k3} the torsion subgroup of $E(\overline{\mathbb{Q}}(t))$ is isomorphic to $\mathbb{Z}/2\mathbb{Z}\oplus\mathbb{Z}/4\mathbb{Z}$ or $(\mathbb{Z}/4\mathbb{Z})^2$. We will show that the latter case does not hold.

We already know that the $2$-torsion points are of the form $(f^2,0),(g^2,0)$ or $(0,0)$. According to Lemma \ref{lemma:torsion_subgroup_general}, to prove the corollary it suffices to show that points $(f^2,0)$ and $(g^2,0)$ are not $2$-divisible. So suppose, to the contrary, that there exists a point $P=(x,y)\in E(\overline{\mathbb{Q}}(t))_{\textrm{tors}}$ such that $2P=(g^2,0)$. The duplication formula for $P$ implies that
\[\frac{(x-f g)^2 (f g+x)^2}{4 x \left(x-f^2\right) \left(x-g^2\right)}-g^2=0\]
or equivalently
\begin{equation}\label{equation:quadratic_4_torsion}
-f^2 g^2 + 2 g^2 x - x^2=0.
\end{equation}
The discriminant of the above quadratic equation equals $-4 (f^2 - g^2) g^2 $. But we have assumed that $f^2-g^2$ is separable, hence it cannot be a square in $\overline{\mathbb{Q}}(t)^{\times}$. So the root $x$ of (\ref{equation:quadratic_4_torsion}) cannot lie in $\overline{\mathbb{Q}}(t)$. A similar argument works in the case $2P=(f^2,0)$. So the group generated by $T_{1}$ and $T_{2}$ as described in Lemma \ref{lemma:torsion_subgroup_general} is equal to $E(\overline{\mathbb{Q}}(t))_{\textrm{tors}}$.

\end{proof}

\begin{example}
If we drop the assumption on separability of $f^2-g^2$, then we can easily find polynomials $f,g$ such that the torsion subgroup of $\overline{\mathbb{Q}}(t)$-rational points on curve  (\ref{eq:general_equation_f_g}) is isomorphic to $(\mathbb{Z}/4\mathbb{Z})^2$. Take
\[f=-1+t^2,\quad g=1+t^2.\]
We use the fact that the elliptic surface attached to curve (\ref{eq:general_equation_f_g}) has bad fibers of type $I_{4}$ and then homomorphism (\ref{equation:injective_torsion_homomorphism}) can be used to show that all $\overline{\mathbb{Q}}(t)$-rational torsion points are of order $1$, $2$ or $4$. Generators of this group are points $T_{2}$ and $T_{3}$, where $T_{2}$ comes from Corollary \ref{corollary:torsion_structure_f_g_curves_with_assumptions} and $T_{3}$ satisfies $2T_{3}=(g^2,0)$.

Observe that $f^2+g^2$ is not a square of any polynomial, so the rank of the Mordell-Weil group $E(\overline{\mathbb{Q}}(t))$ is zero as predicted by \cite[Thm. 2.2]{Cox_rational_k3}.
\end{example}

\subsection{Points of infinite order}\label{sub:heights}
For an elliptic curve $E$ over $K=\overline{\mathbb{Q}}(t)$ we denote by $\langle \cdot,\cdot\rangle_{E}$ the height pairing attached to $E$ as in \cite{Shioda_Mordell_Weil}. The group $E(K)/E(K)_{\textrm{tors}}$ with the induced pairing $\langle\cdot,\cdot\rangle_{E}$ is a positive definite lattice, cf. \cite[Theorem 7.4]{Shioda_Mordell_Weil}. To simplify the notation, we write $\langle\cdot,\cdot\rangle$ if the curve $E$ is fixed. Explicitly, for two points $P,Q\in E(K)$ their intersection pairing is given by
\begin{equation}\label{equation:height_pairing_formula}
\langle P,Q\rangle = \chi(S)+\overline{P}.\overline{O}+\overline{Q}.\overline{O}-\overline{P}.\overline{Q}-\sum_{v\in B}c_{v}(P,Q). 
\end{equation}
For a point $P$ in $E(K)$ we denote by $\overline{P}$ the curve which lies in $S$ and is the image of a section determined by point $P$, cf. \cite[Lemma 5.2]{Shioda_Mordell_Weil}. The curve $\overline{O}$ is the image of the zero section $O:\mathbb{P}^{1}_{\overline{\mathbb{Q}}}\rightarrow S$. In the case $P=Q$ the formula simplifies to
\begin{equation}\label{equation:height_formula}
\langle P,P\rangle = 2\chi(S)+2\overline{P}.\overline{O}-\sum_{v\in B}c_{v}(P,P). 
\end{equation}
The rational numbers $c_{v}(P,Q)$ depend on the fiber type above $v\in B$ and on the points $P$ and $Q$, cf. \cite[Theorem 8.6]{Shioda_Mordell_Weil}. We will need explicit version of this theorem only for the case when fibers are of type $I_{n}$. We call $\langle P,P\rangle$ the \textit{height of point} $P$.

\begin{lemma}\label{lemma:height_correction_for_I_N}
Let $f,g\in\overline{\mathbb{Q}}[t]$ be two coprime polynomials. Let $E$ be an elliptic curve with Weierstrass equation (\ref{eq:general_equation_f_g}). Let $B$ be the set of points in $\mathbb{P}^{1}_{\overline{\mathbb{Q}}}(\overline{\mathbb{Q}})$ over which the Kodaira-N\'{e}ron model of $E$ has bad reduction. We have the following equality
\begin{equation}\label{equation:intersection_formula_P_O}
\overline{P}.\overline{O}=-\frac{1}{2}\sum\limits_{a\in\mathbb{P}^{1}_{\overline{\mathbb{Q}}}(\overline{\mathbb{Q}})}\min\{0,v_{a}(x)\}.
\end{equation}
If the curve $\overline{P}$ intersects the same component as $\overline{O}$ in the fiber over $a\in B$, then we put $c_{a}(P,P)=0$. Otherwise, let $n=\min\{v_{a}(y),v_{a}(\Delta)/2\}$, where $\Delta$ is a discriminant of the equation (\ref{eq:general_equation_f_g}). Then
\begin{equation}\label{equality:formula_for_c_a}
c_{a}(P,P)=\frac{n(N-n)}{N}
\end{equation}
if the fiber above $a$ is of type $I_{N}$.
\end{lemma}
\begin{proof}
Let $P$ be a point in $E(\overline{\mathbb{Q}}(t))$. The formula (\ref{equation:intersection_formula_P_O}) follows from \cite[III, \S 9]{Silverman_book}. Equality (\ref{equality:formula_for_c_a}) comes from the proof of \cite[Theorem 8.6]{Shioda_Mordell_Weil}, \cite[\S 1]{Cremona_heights} and \cite[\S 5, (28)]{Silverman_heights}.
\end{proof}

\begin{lemma}\label{lemma:rank_2_subgroup_E}
Let $f,g\in\overline{\mathbb{Q}}[t]$ be two coprime polynomials and assume there exists another polynomial $h\in\overline{\mathbb{Q}}[t]$ such that $f^2+g^2=h^2$. Moreover, let $\deg g\leq \deg f$ and $\deg f>0$. For the elliptic curve $E$ over $\overline{\mathbb{Q}}(t)$ determined by the equation (\ref{eq:general_equation_f_g}) the following points
\begin{align*}
Q_{1}=&(-g^2,\sqrt{-2}g^2 h),\\
Q_{2}=&(h^2,fgh)
\end{align*}
are of infinite order and $\langle Q_{1},Q_{1}\rangle = \deg f$, $\langle Q_{2},Q_{2}\rangle=2\deg f$. Group generated by the points $Q_{1},Q_{2}$ has rank 2.
\end{lemma}
\begin{proof}
Let $\mathcal{E}=(S,\mathbb{P}^{1}_{\overline{\mathbb{Q}}},\pi)$ be the Kodaira-N\'{e}ron model of $E$. Euler characteristic $\chi(S)$ equals $\deg f$ by Lemma \ref{lemma:bad_fibres_types_f_g_curve}. In the $s$-coordinate chart the point $Q_{1}$ looks as follows
\begin{equation}\label{equation:model_Q_1_at_infinity}
Q_{1}=(-g(1/s)^2 s^{2\deg f}, \sqrt{-2}g(1/s)^{2}h(1/s)s^{3\deg f} ).
\end{equation}
For the valuation at infinity $v_{\infty}=v_{s}$ we get $v_{s}(x(Q_{1}))=2\deg f-2\deg g\geq 0$. For the valuation $v_{a}$ at a finite place $a\in\overline{\mathbb{Q}}$ we consider $Q_{1}$ as a point with coordinates expressed with coordinate $t$. Observe that $v_{a}(x(Q_{1}))=v_{a}(-g^2)\geq 0$. Equality (\ref{equation:intersection_formula_P_O}) applied to $Q_{1}$ implies that
\[\overline{Q_{1}}.\overline{O}=0.\]
Assume that $B$ is the set of points in $\mathbb{P}^{1}_{\overline{\mathbb{Q}}}(\overline{\mathbb{Q}})$ for which $\mathcal{E}$ has singular fibers. According to Lemma \ref{lemma:bad_fibres_types_f_g_curve} element $a$ lies in $B$ if and only if $(f\cdot g\cdot (f^2-g^2))(a)=0$ or $a=\infty$. The curve $\overline{Q_{1}}$ intersects the same component as $\overline{O}$ in the fiber above $a\in B\setminus\{\infty\}$ only when $g(a)\neq 0$. For $a\in B\setminus\{\infty\}$ such that $g(a)=0$ the point $Q_{1}$ reduces to $(0,0)$ and the curve $\overline{Q_{1}}$ intersects other component that the curve $\overline{O}$ in the fiber above $a$. Observe that $v_{a}(y(Q_{1}))=v_{a}(\sqrt{-2}g^2h)=\frac{1}{2}v_{a}(\Delta)$. Put $n=\min\{v_{a}(y(Q_{1})),v_{a}(\Delta)/2\}$, so $n=2v_{a}(g)$. The fiber above $a$ has Kodaira type $I_{4v_{a}(g)}$, so in the notation of Lemma \ref{lemma:height_correction_for_I_N} there are $N=4v_{a}(g)$ components. 
Then equation (\ref{equality:formula_for_c_a}) gives 
\[c_{a}(P,P)=v_{a}(g).\]
If $a=\infty$ lies in $B$ we use equation (\ref{equation:model_Q_1_at_infinity}). In that case $\overline{Q_{1}}$ does not intersect the same component of the fiber $\pi^{-1}(\infty)$ as $\overline{O}$ exactly when $2\deg f>2 \deg g$ and $3\deg f >2\deg g+\deg h$. But from $\deg f>\deg g$ we get $\deg h=\deg f$ because of $f^2+g^2=h^2$. The fiber above $\infty$ has $N=12\deg f-\deg\Delta$ components. But $N$ equals $4(\deg f-\deg g)$. Then from the definition of $n$ we get
\[n=\min\{2(\deg f-\deg g),2(\deg f-\deg g)\}=2(\deg f-\deg g).\]
We apply again the formula (\ref{equality:formula_for_c_a}) to obtain $c_{a}(P,P)=\deg f-\deg g$. Finally, the height of $Q_{1}$ can be computed
\[\langle Q_{1},Q_{1}\rangle = 2\deg f -\deg g-(\deg f-\deg g)=\deg f.\]
In order to compute the height of $Q_{2}$ we express point $Q_{2}$ in the local coordinate system of $s$
\[Q_{2}=(h(1/s)^2 s^{2\deg f}, f(1/s)g(1/s)h(1/s)s^{3\deg f} ).\]
Similar reasoning to that performed for $Q_{1}$ shows that $\overline{Q_{2}}.\overline{O}=0$. The polynomials $h^2, fgh$ and $\Delta$ are pairwise coprime. So for $a\in B\setminus\{\infty\}$ the curve $\overline{Q_{2}}$ intersects the same component as $\overline{O}$ in the fiber above $a$. When $a=\infty$, then $\overline{Q_{2}}$ does not intersect the same component as $\overline{O}$ exactly when $2\deg f>2\deg h$ and $2\deg f>\deg g+\deg h$. The first inequality implies that $\deg f=\deg g$ by the definition of $h$. Moreover, if $f^2=a_{m}t^{m}+\ldots$ and $g^2=b_{m}t^{m}+\ldots$, then $a_{m}=-b_{m}\neq 0$. Hence $f^2-g^2=2a_{m}t^{m}+\ldots$ and $\deg(f^2-g^2)=\deg(f^2)=2\deg f$. The fiber above $a$ has
\[N=12\deg f-\deg\Delta = 4\deg f-2\deg(f^2-g^2)\]
components. This means that if $\overline{Q_{2}}$ would intersect a different component than $\overline{O}$ in the fiber above $a$, then $N$ would equal $0$, which contradicts the fact that each fiber has at least $1$ component. Hence $c_{\infty}(Q_{2},Q_{2})=0$ and by the height formula
\[\langle Q_{2},Q_{2}\rangle = 2\deg f.\]
Since $\deg f>0$, the points $Q_{1}$ and $Q_{2}$ are of infinite order, because their heights are positive. To prove that they are linearly independent it suffices to show that the Gram matrix with respect to $\langle\cdot,\cdot\rangle$ is nonzero. We check that
\begin{equation}
\det\left(\begin{array}{cc}
\langle Q_{1},Q_{1}\rangle & \langle Q_{1},Q_{2}\rangle\\
\langle Q_{2},Q_{1}\rangle & \langle Q_{2},Q_{2}\rangle\\
\end{array}\right)=2(\deg f)^2-\langle Q_{1},Q_{2}\rangle^2.
\end{equation}
The equality $2(\deg f)^2-\langle Q_{1},Q_{2}\rangle^2=0$  is impossible because $\langle Q_{1},Q_{2}\rangle$ is a rational number and $\deg f>0$. This finishes the proof of the theorem.
\end{proof}

\begin{corollary}\label{corollary:rank_2_for_f_g_curves}
Assumptions and notation as in Lemma \ref{lemma:rank_2_subgroup_E}. The points
\begin{align*}
P_{1}=&(-(1+\sqrt{2})g(g-h),\sqrt{-1}(1+\sqrt{2})g(g-h)(\sqrt{2}g-h)),\\
P_{2}=&((f-h)(g-h),(f+g)(f-h)(g-h))
\end{align*}
are of infinite order and linearly independent in $E(\overline{\mathbb{Q}}(t))$. The Gram matrix with respect to $\langle , \rangle$ has the form
\begin{equation}\label{equation:Gram_matrix_P1_P2}
\left(\begin{array}{cc}
\langle P_{1},P_{1}\rangle & \langle P_{1},P_{2}\rangle\\
\langle P_{2},P_{1}\rangle & \langle P_{2},P_{2}\rangle\\
\end{array}\right)=\left(\begin{array}{cc}
\frac{1}{4}\deg f & 0\\
0 & \frac{1}{2}\deg f\\
\end{array}\right)
\end{equation}
and 
\begin{align}
Q_{1}=&-2P_{1}\label{align:Q_1_equality},\\
Q_{2}=&-2P_{2}\label{align:Q_2_equality}.
\end{align}
\end{corollary}

\begin{proof}
The equalities (\ref{align:Q_1_equality}) and (\ref{align:Q_2_equality}) follow from $f^2+g^2=h^2$ and a simple direct computation. Lemma \ref{lemma:rank_2_subgroup_E} implies that $\langle Q_{1},Q_{1}\rangle = \deg f$. The form $\langle\cdot,\cdot\rangle$ is bilinear hence $\langle P_{1},P_{1}\rangle =\frac{1}{4}\deg f$. In the same way we obtain $\langle P_{2},P_{2}\rangle=\frac{1}{2}\deg f$. By the assumption $\deg f>0$ the heights of $P_{1}$ and $P_{2}$ are positive so the points are of infinite order. The determinant of Gram matrix (\ref{equation:Gram_matrix_P1_P2}) equals $\frac{1}{8}(\deg f)^2-\langle P_{1},P_{2}\rangle^2$ and cannot be zero since the number $\langle P_{1},P_{2}\rangle$ is rational. Therefore, the points $P_{1}$ and $P_{2}$ are linearly independent. We shall prove now that $\langle P_{1},P_{2}\rangle=0$. Using the fact that (\ref{align:Q_1_equality}) and (\ref{align:Q_2_equality}) hold, we have to prove that $\langle Q_{1},Q_{2}\rangle=0$. Again, by bilinearity of the form $\langle\cdot,\cdot\rangle$ we get
\[\langle Q_{1}+Q_{2},Q_{1}+Q_{2}\rangle =\langle Q_{1},Q_{1}\rangle +\langle Q_{2},Q_{2}\rangle +2\langle Q_{1},Q_{2}\rangle= 3\deg f+2\langle Q_{1},Q_{2}\rangle. \]
So equation $\langle Q_{1},Q_{2}\rangle=0$ is true if and only if $\langle Q_{1}+Q_{2},Q_{1}+Q_{2}\rangle=3\deg f$ holds. We will prove this last equality. In explicit terms, let $Q:=Q_{1}+Q_{2}=(x,y)$ where
\begin{equation}
x=\frac{-f^4-2 \sqrt{-2} f^3 g+3 f^2 g^2+g^4}{\left(f+\sqrt{-2}g\right)^2},
\end{equation}
\begin{equation}
y=\frac{g h \left(-f^4-\sqrt{-2} f^3 g-2 f^2 g^2-4 \sqrt{-2} f g^3+3 g^4\right)}{\left(f+\sqrt{-2} g\right)^3}.
\end{equation}
To simplify the notation we label three polynomials
\begin{align*}
F_{1}:= & -f^4-2 \sqrt{-2} f^3 g+3 f^2 g^2+g^4,\\
F_{2}:= & g h \left(-f^4-\sqrt{-2} f^3 g-2 f^2 g^2-4 \sqrt{-2} f g^3+3 g^4\right),\\
H:= & f+\sqrt{-2}g.
\end{align*}
Let $\mathcal{E}=(S,\mathbb{P}^{1}_{\overline{\mathbb{Q}}},\pi)$ be the Kodaira-N\'{e}ron model for $E$. The set $B\subset\mathbb{P}^{1}_{\overline{\mathbb{Q}}}(\overline{\mathbb{Q}})$ is such that $a\in B$ if and only if $\pi^{-1}(a)$ is not smooth. In other words $a\in B$ if and only if $(f\cdot g\cdot (f^2-g^2))(a)=0$ or $a=\infty$. Let us first compute the intersection number $\overline{Q}.\overline{O}$. The function $x$ has a pole at $a$ when $a\in\overline{\mathbb{Q}}$. But if $H(a)=0$, then $F_{1}(a)\neq 0$, so $v_{a}(x)=-2\deg H$. When $a=\infty$, we express $Q$ in the coordinate system of $s=1/t$
\[Q=\left(\frac{F_{1}(1/s)}{(H(1/s))^2} s^{2\deg f}, \frac{F_{2}(1/s)}{(H(1/s))^3} s^{3\deg f} \right).\]
When $\deg f> \deg g$, then $\deg F_{1}=4\deg f$ and $\deg H=\deg f$. In consequence $v_{s}(x)=0$.

\medskip\noindent
When $\deg f= \deg g$ we split the computations into two separate cases.

\medskip\noindent
$\mathbf{1^{\circ}}$ Let $\deg H<\deg f$. If $f=a_{m}t^m+\ldots$ and $g=b_{m}t^{m}+\ldots$ with $a_{m}\neq 0, b_{m}\neq 0$, then $a_{m}+\sqrt{-2}b_{m}=0$. The leading coefficient of the polynomial $F$ is
\[A:=-a_{m}^4-2 \sqrt{-2} a_{m}^3 b_{m}+3 a_{m}^2 b_{m}^2+b_{m}^4.\]
Because of the equality $a_{m}=-\sqrt{-2}b_{m}$, the coefficient $A$ is nonzero, so $\deg F_{1}=4\deg f$. So the valuation $v_{s}(x)$ equals $-(2\deg f-2\deg H)$.

\medskip\noindent
$\mathbf{2^{\circ}}$ Let $\deg H=\deg f$. In this case we have $\deg F_{1}\leq 4\deg f$ and $v_{s}(x)=4\deg f-\deg F_{1}\geq 0$. Application of formula (\ref{equation:intersection_formula_P_O}) gives
\[\overline{Q}.\overline{O}=\deg f.\]
It remains to prove that $\sum\limits_{a\in B}c_{a}(Q,Q)=0$. Observe first that for $a\in B\setminus\{\infty\}$ value $x(a)$ is finite and nonzero, so by definition $c_{a}(Q,Q)=0$. If $a=\infty$, then on the basis of previous computations we get
\begin{itemize}
\item[(i)] if $\deg f>\deg g$, then $v_{\infty}(x)=v_{s}(x)=0$, so $c_{\infty}(Q,Q)=0$,
\item[(ii)] if $\deg f=\deg g$ and $\deg H<\deg f$, then $v_{\infty}(x)<0$, so $c_{\infty}(Q,Q)=0$,
\item[(iii)] if $\deg f=\deg g$ and $\deg H=\deg f$, then $v_{\infty}(x)\geq 0$ and $v_{\infty}(y)\geq 0$.
\end{itemize}
We continue the argument from point (iii). If the curve $\overline{Q}$ would not intersect the same component as $\overline{O}$ in the fiber above $a$, then we would have $v_{a}(x)>0$ and $v_{a}(y)>0$. But this is possible only if $\deg F_{1} < 4\deg f$ and $\deg F_{2}<6\deg f$. First inequality means that $A=0$ and the second implies
\[B:=b_{m}\sqrt{a_{m}^{2}+b_{m}^{2}}\left(-a_{m}^4-\sqrt{-2} a_{m}^3 b_{m}-2 a_{m}^2 b_{m}^2-4 \sqrt{-2} a_{m} b_{m}^3+3 b_{m}^4\right)=0.\]
We shall prove now that $A=0$ and $B^2=0$ cannot both occur. If we treat $a_{m}$ and $b_{m}$ as variables $a$ and $b$ respectively, then we are looking for a solution in variables $a,b,s$ of the following system
\begin{align*}
b^2 (a^2 + b^2) (-a^4 - 2 a^2 b^2 + 3 b^4 - a^3 b s - 4 a b^3 s)^2 &=0\\
-a^4 + 3 a^2 b^2 + b^4 - 2 a^3 b s &=0\\
2 + s^2&=0
\end{align*}
If $b=0$, then $a=0$, but this means $a_{m}=b_{m}=0$, which contradicts the assumption that $a_{m}$ and $b_{m}$ are the leading coefficients of $f$ and $g$. If $a^2=-b^2$, then the second equation of the system reduces to $-b^3 (3 b - 2 a s)=0$ and again we obtain $a=b=0$, hence a contradiction. It remains to consider the case when $-a^4 - 2 a^2 b^2 + 3 b^4 - a^3 b s - 4 a b^3 s=0$. Running MAGMA package \cite{Magma} we can check that the ideal
\[I=(-a^4 - 2 a^2 b^2 + 3 b^4 - a^3 b s - 4 a b^3 s,-a^4 + 3 a^2 b^2 + b^4 - 2 a^3 b s,s^2+2)\subset\mathbb{Q}[a,b,s]\]
can be expressed in a different way in terms of a Gr\"{o}bner basis
\begin{equation*}
\begin{split}
I=(a^4 + 7a^2b^2 + 8ab^3s - 5b^4,a^3 b + \frac{5}{2} a^2  b^2 s - 4 a b^3 - b^4 s,a^2 b^3 + \\
 \frac{2}{3} a b^4 s,a b^5 + \frac{9}{14} b^6 s,b^7,s^2 + 2)\subset \mathbb{Q}[a,b,s].
\end{split}
\end{equation*}
This implies that the system
\begin{align*}
-a^4 - 2 a^2 b^2 + 3 b^4 - a^3 b s - 4 a b^3 s &=0\\
-a^4 + 3 a^2 b^2 + b^4 - 2 a^3 b s &=0\\
2 + s^2&=0
\end{align*}
is equivalent to $a=b=s^2+2=0$. By virtue of (\ref{equation:height_formula}) we finally get the desired equality
\[\langle Q,Q\rangle =3\deg f.\]
\end{proof}

The following proposition gives a characterization of triples $(f,g,h)$ of coprime polynomials in $\overline{\mathbb{Q}}[t]$ such that $f^2+g^2=h^2$. It will be used in the proof of Theorem \ref{theorem:Main_theorem_geometric_MW}.
\begin{proposition}\label{proposition:h_1_h_2_generators}
Let $f,g,h\in\overline{\mathbb{Q}}[t]$ be the polynomials such that $f^2+g^2=h^2$. If $f,g$ are coprime, then there exist two polynomials $h_{1},h_{2}\in\overline{\mathbb{Q}}[t]$ that are coprime and
\begin{align}
f &= \frac{1}{2}(h_{1}^{2}+h_{2}^{2}),\label{align:wielomian_f}\\
g &= \frac{1}{2i}(h_{1}^{2}-h_{2}^{2}), \quad i=\sqrt{-1},\label{align:wielomian_g}\\
h &= h_{1}h_{2}.\label{align:wielomian_h}
\end{align}
The other way around, let $h_{1},h_{2}\in\overline{\mathbb{Q}}[t]$ be two coprime polynomials. We construct polynomials $f,g$ and $h$ determined by the formulas (\ref{align:wielomian_f}), (\ref{align:wielomian_g}), (\ref{align:wielomian_h}). They satisfy the relation $f^2+g^2=h^2$ and $f,g$ are coprime.
\end{proposition}
\begin{proof}
Observe that $f^2+g^2=(f+\sqrt{-1}g)(f-\sqrt{-1}g)=h^2$. The ring $\overline{\mathbb{Q}}[t]$ is a unique factorization domain, so both factors on the left-hand side are squares, say $(f+\sqrt{-1}g)=h_{1}^{2}$ and $(f-\sqrt{-1}g)=h_{2}^{2}$. The rest of the proof follows.
\end{proof}

\begin{proof}[Proof of Theorem \ref{theorem:Main_theorem_geometric_MW}]
By Proposition \ref{proposition:h_1_h_2_generators} we have two polynomials $h_{1}$ and $h_{2}$ that satisfy equations \eqref{align:wielomian_f}, \eqref{align:wielomian_g} and \eqref{align:wielomian_h}. If the polynomial $g$ would be constant then this will imply that $h_{1}$ and $h_{2}$ are constant too, a contradiction. So in fact we have the bound $1\leq \deg g$. Similarly, the polynomial $h$ is non-constant and of degree at most $2$. It is also separable, because otherwise $\deg f=4$ or $f$ would have a common root with $g$.

\medskip\noindent
The polynomial $f^2-g^2$ equals $\frac{1}{2}(h_{1}^4+h_{2}^4)$. The degree $\deg h$ is at most $2$, so consider first the case $\deg h=1$, hence $h_{1}=a\in\overline{\mathbb{Q}}^{\times}$ and $h_{2}=c(t-d)$, $c\neq 0$. The expression $\frac{1}{2}(h_{1}^4+h_{2}^4)$ written as a polynomial of variable $t$ is of degre $3$ or $4$. In both cases it has nonzero discriminant, so it is a separable polynomial. When $\deg h=2$, then let $h_{1}=a(t-b)$ and $h_{2}=c(t-d)$, $a,c\neq 0$ and $b\neq d$. Again, the discriminant of $\frac{1}{2}(h_{1}^4+h_{2}^4)$ written as a polynomial of $t$ is nonzero, hence $f^2-g^2$ is separable. A similar argument proves that $g$ is separable when $\deg g=2$ and also that $f$ is separable.

\medskip\noindent
Let $\mathcal{E}=(S,\mathbb{P}^{1}_{\overline{\mathbb{Q}}},\pi)$ be the Kodaira-N\'{e}ron model of $E$. We denote by $N_{\infty}$ the number of components in $\pi^{-1}(\infty)$. Lemma \ref{lemma:bad_fibres_types_f_g_curve} implies that $N_{\infty}=16-4\deg g-2\deg(f^2-g^2)$. Let $r$ denote the rank of  $E(\overline{\mathbb{Q}}(t))$. Application of Shioda-Tate formula \cite[Corollary 5.3]{Shioda_Mordell_Weil} allows us to compute the rank $\rho(S)$ of the N\'{e}ron-Severi group associated with $S$
\begin{equation}\label{equation:Picard_number_bound_f_g_curve}
\rho(S)=8+\deg(f^2-g^2)+3\deg g+r+\max\{N_{\infty}-1,0\}.  
\end{equation}
From the assumption we have $\chi(S)=\deg f=2$ and by the Lefschetz (1,1)-classes theorem \cite[Prop. 3.3.2]{Huybrechts_Geometry} we get the bound $\rho(S)\leq 20$. Combining this with equality (\ref{equation:Picard_number_bound_f_g_curve}) we obtain the bound for the rank
\[r\leq 12-\max\{N_{\infty}-1,0\}-3\deg g-\deg(f^2-g^2).\]
If $\deg g=1$, then $\deg(f^2-g^2)=4$ and $N_{\infty}=4$ and $r\leq 2$. If $\deg g=2$, then we consider two cases. For $\deg(f^2-g^2)=4$ we have $N_{\infty}=0$ (the fiber above $\infty$ is not singular) and again $r\leq 2$. For $\deg(f^2-g^2)<4$ we have $N_{\infty}=8-2\deg(f^2-g^2)>0$, hence
\[r\leq 12-(7-2\deg(f^2-g^2))-6-\deg(f^2-g^2)\]
\[r+1\leq\deg(f^2-g^2)\]
We combine this with Corollary \ref{corollary:rank_2_for_f_g_curves} to get $\deg(f^2-g^2)=3$. This means that again $r\leq 2$ and application of Corollary \ref{corollary:rank_2_for_f_g_curves} finishes the proof.

\medskip\noindent
The structure of the torsion subgroup in $E(\overline{\mathbb{Q}}(t))$ has been established in Corollary \ref{corollary:torsion_structure_f_g_curves_with_assumptions}. Let us prove that $P_{1}$ and $P_{2}$ generate the free part of the group $E(\overline{\mathbb{Q}}(t))$. From now on let $K$ denote $\overline{\mathbb{Q}}(t)$. The pair $(E(K)/E(K)_{\textrm{tors}},\langle\cdot,\cdot\rangle_{E})$ is a positive definite lattice. Elliptic surface $\mathcal{E}$ admits only singular fibers of types $I_{2}$ and $I_{4}$ so by the formulas (\ref{equation:height_pairing_formula}) and (\ref{equality:formula_for_c_a}) we get $\langle P,Q\rangle_{E}\in\frac{1}{4}\mathbb{Z}$ for any $P,Q\in E(K)/E(K)_{\textrm{tors}}$. Corollary \ref{corollary:rank_2_for_f_g_curves} implies that $\langle P_{1},P_{1}\rangle_{E}=1/2$, $\langle P_{2},P_{2}\rangle_{E}=1$ and $\langle P_{1},P_{2}\rangle_{E}=\langle P_{2},P_{1}\rangle_{E}=0$. To keep the values of the pairing integral we define a lattice $(\Lambda,\langle\cdot,\cdot\rangle)$ such that $\Lambda=E(K)/E(K)_{\textrm{tors}}$ and $\langle\cdot,\cdot\rangle=4\langle\cdot,\cdot\rangle_{E}$. We present an argument similar to the proof of
\cite[Lemma 6.5]{Naskrecki_Acta}. Let $\Lambda'$ denote a sublattice in $\Lambda$ spanned by $P_{1}$ and $P_{2}$. Our goal is to prove that $\Lambda=\Lambda'$. For this, let $n=[\Lambda:\Lambda']$ be the index of $\Lambda'$ in $\Lambda$. With respect to $\langle\cdot,\cdot\rangle$ we have
\[\langle P_{i},P_{j}\rangle = \left\{\begin{array}{cc} 2i, & \textrm{if }i=j,\\ 0, & \textrm{if }i\neq j.   \end{array}\right.\]
Hence the discriminant $\Delta(\Lambda')$ equals $8$, so
\[8=\Delta(\Lambda')=n^{2}\Delta(\Lambda).\]
This means that $n^{2}\mid 8$ and therefore $n\mid 2$. Let $G=E(K)$. The subgroup generated by $P_{1},P_{2}$, $T_{1}$ and $T_{2}$ is denoted by $H$. The index $[G:H]$ is equal to $n$. By the stacked basis theorem for abelian groups there exists two elements $R_{1},R_{2}\in G$ such that $R_{1},R_{2},T_{1},T_{2}$ generate $G$ and $H$ is generated by $aR_{1}, bR_{2},T_{1},T_{2}$ where $a,b\in\mathbb{Z}$, $a\mid b$ and $ab=n$. By \cite[X, Proposition 1.4]{Silverman_arithmetic} there exists an injective homomorphism
\[\psi:E(K)/2E(K)\hookrightarrow K^{\times}/(K^{\times})^{2}\times K^{\times}/(K^{\times})^{2}\]
such that if $(x,y)$ is a point in $E(K)\setminus E(K)[2]$ the image $\psi(x,y)$ is as follows
\[\psi(x,y)=(x,x-f^2).\]
Let $\phi:G\rightarrow G/2G$ denote the map $\phi(x)=x+2G$ and $\eta=\psi\circ\phi$. Group $E(K)$ is of rank $2$ and by Corollary \ref{corollary:torsion_structure_f_g_curves_with_assumptions} torsion subgroup of $E(K)$ is isomorphic to $\mathbb{Z}/2\mathbb{Z}\oplus\mathbb{Z}/4\mathbb{Z}$, hence $G/2G\cong (\mathbb{Z}/2\mathbb{Z})^{4}$. The homomorphism $\eta$ is injective so $\eta(G)\cong (\mathbb{Z}/2\mathbb{Z})^{4}$. Suppose $n=2$, then $a=1$ and $b=2$ and it follows that $\eta(H)\cong (\mathbb{Z}/2\mathbb{Z})^{3}$. So if we prove that $\eta(H)\cong (\mathbb{Z}/2\mathbb{Z})^{4}$, then $n=1$ and the proof will be finished. Let $\zeta=e^{2\pi \sqrt{-1}/8}$ be a primitive root of unity of degree $8$. Let $h_{1}$ and $h_{2}$ be two polynomials in $\overline{\mathbb{Q}}[t]$ defined in Proposition \ref{proposition:h_1_h_2_generators}. By a simple computation and using the fact that coordinates of elements $\eta(P_{1}),\eta(P_{2})$ and $\eta(T_{2})$ are determined up to squares in $K$ we compute
\begin{align*}
\eta(P_{1})=&\left((h_1 - h_2)  (h_1 + h_2), (\zeta h_1 + h_2) (\zeta^{3} h_1 +h_2)\right),\\
\eta(P_{2})=&\left(1, (\zeta h_1 + h_2) (\zeta^{5} h_1 + h_2)\right),\\
\eta(T_{1})=&\left(1,(\zeta h_1 + h_2)(\zeta^3 h_1 + h_2) (\zeta^{5} h_1 + h_2)(\zeta^7 h_1 + h_2)\right),\\
\eta(T_{2})=&\left( (h_1 - h_2) (h_1 + h_2) (h_1^2 + h_2^2), (\zeta^{2}h_1  + h_2)(\zeta^{3} h_1  + h_2)(\zeta^{6}h_1  + h_2)(\zeta^{7}h_1  + h_2) \right).
\end{align*}
The first coordinate of $\eta(P_{1})$ is equal to $g$ up to squares, so by assumption on $g$ it cannot be equal to $1$ modulo $(K^{\times})^{2}$. The element $\eta(P_{1})$ has order $2$. We recall that $h_{1}=a(t-b)$, $h_{2}=c(t-d)$ for some $a,b,c,d\in\overline{\mathbb{Q}}$, $a,c\neq 0$ or $h_{1}=a\neq 0$ and $h_{2}=c(t-d)$. Below we assume that the former case holds, for the latter the computations are very similar. The polynomials $f$ and $g$ have no common roots, so $b\neq d$. The second coordinate $\eta(P_{2})_{2}$ of $\eta(P_{2})$ equals
\[(-\sqrt{-1}) a^2 b^2+c^2 d^2+t \left(-2 c^2 d+2 \sqrt{-1} a^2 b\right)+t^2 \left(c^2-\sqrt{-1} a^2\right).\]
We easily check that under our assumptions the polynomial has degree at least $1$. If it is of degree $1$ then it cannot be a square. Assume that it has degree $2$. Discriminant of this polynomial with respect to $t$ is $4 \sqrt{-1} a^2 c^2 (b-d)^2$, hence $\eta(P_{2})_{2}\not\equiv 1 (\textrm{mod }(K^{\times})^{2})$ and the order of $\eta(P_{2})$ is $2$. The first coordinate $\eta(T_{2})_{1}$ of the point $\eta(T_{2})$ is a polynomial
\[a^4 b^4-c^4 d^4+t \left(4 c^4 d^3-4 a^4 b^3\right)+t^2 \left(6 a^4 b^2-6 c^4 d^2\right)+t^3 \left(4 c^4 d-4 a^4 b\right)+t^4 \left(a^4-c^4\right).\]
If it would be of degree $3$ or $1$, we are done. It cannot be constant or of degree $2$ since $c\neq 0$ and $b\neq d$. Assume that it has degree four. Its discriminant $-256 a^{12} c^{12} (b-d)^{12}$ is nonzero, so $\eta(T_{2})_{1}\not\equiv 1 (\textrm{mod }(K^{\times})^{2})$ and $\eta(T_{2})$ is of order $2$. By the preceding arguments we already know that $\eta(P_{1})\neq\eta(P_{2})$ and $\eta(P_{2})\neq\eta(T_{2})$. Up to squares $h_{1}^{2}+h_{2}^{2}$ is the same as $f$, so $\eta(P_{1})\neq \eta(T_{2})$. This means that $\eta(P_{1})$ and $\eta(P_{2})$ span a group isomorphic to $(\mathbb{Z}/2\mathbb{Z})^{2}$. With the element $\eta(T_{2})$ they span a group isomorphic to $(\mathbb{Z}/2\mathbb{Z})^{3}$ if and only if $\eta(P_{1})\cdot\eta(P_{2})\neq\eta(T_{2})$. The last inequality is satisfied because $f$ is separable, so $h_{1}^{2}+h_{2}^{2}$ is not a square in $K$. Now we have to show that $\eta(T_{1})$ is not contained in the subgroup spanned by the elements $\eta(P_{1}),\eta(P_{2})$ and $\eta(T_{2})$. The element $\eta(T_{1})$ is of order $2$ if and only if the polynomial
\[a^4 b^4+c^4 d^4+t \left(-4 a^4 b^3-4 c^4 d^3\right)+t^2 \left(6 a^4 b^2+6 c^4 d^2\right)+t^3 \left(-4 a^4 b-4 c^4 d\right)+t^4 \left(a^4+c^4\right)\]
is not a square. It cannot be of degree less than $3$. We compute its discriminant under the assumption that it has degree $4$. It equals $256 a^{12} c^{12} (b-d)^{12}$, hence the polynomial is separable under our assumptions, hence $\eta(T_{1})$ has order $2$. Our previous computations already show that $\eta(T_{1})$ is not contained in the set 
\[\{\eta(P_{1}),\eta(T_{2}),\eta(P_{1})\eta(P_{2}),\eta(P_{1})\eta(T_{2}),\eta(P_{2})\eta(T_{2}),\eta(P_{1})\eta(P_{2})\eta(T_{2}).\}\]
We need only to check that $\eta(T_{1})\neq\eta(P_{2})$ but this is equivalent to proving that $\sqrt{-1} h_{1}^{2}+h_{2}^{2}$ is not a square in $K^{\times}$. As a polynomial of $t$ it has the form
\[\sqrt{-1} a^2 b^2+c^2 d^2+t \left(-2 c^2 d-2 \sqrt{-1} a^2 b\right)+t^2 \left(c^2+\sqrt{-1} a^2\right)\]
It cannot be a constant polynomial and if it has degree $2$ then its discriminant equals $-4 \sqrt{-1} a^2 c^2 (b-d)^2$ and is nonzero. This implies the polynomial is separable and finishes the proof.
\end{proof}
\begin{remark}
Condition $\deg f=2$ from Theorem \ref{theorem:Main_theorem_geometric_MW} is necessary to get the desired upper bound for the rank. In general if $f$, $g$ and $f^2-g^2$ are separable and $\deg(f^2-g^2)=2\deg f$ we get $r\geq 2$ and
\[2+3\deg f+3\deg g +2\deg f+r+\max\{4(\deg f-\deg g)-1,0\}\leq 10\deg f.\]
If $\deg f=\deg g$, then
\[4\leq 2+r\leq 2\deg f.\]
So the upper bound equals lower bound only when $\deg f=\deg g=2$. In the case $\deg f>\deg g$ we get
\[3\leq 1+r\leq \deg g+\deg f.\]
In this situation again we need $\deg f=2$ to match the lower and upper bound.
\end{remark}
\subsection{Mordell-Weil group over $\mathbb{Q}(t)$}
We compute now the structure of the Mordell-Weil group of the curve (\ref{equation:Weierstrass_equation_f_g_h}) in the situation where all polynomials are defined over $\mathbb{Q}[t]$. This is a generalization of \cite[Lemma 6.5]{Naskrecki_Acta}.
\begin{corollary}\label{corollary:Generic_rank_1_Q_t}
Assumptions and notation as in Theorem \ref{theorem:Main_theorem_geometric_MW}. In addition, let $f,g,h$ lie in $\mathbb{Q}[t]$. Then $E$ is defined over $\mathbb{Q}(t)$ and 
\[E(\mathbb{Q}(t))\cong\mathbb{Z}\oplus\mathbb{Z}/2\mathbb{Z}\oplus\mathbb{Z}/2\mathbb{Z}.\]
Group $E(\mathbb{Q}(t))$ is generated by the points $P_{2},T_{1},2T_{2}$.
\end{corollary}
\begin{proof}
Let $H=E(\mathbb{Q}(t))$ and $G=E(\overline{\mathbb{Q}}(t))$. There exists a natural action of the absolute Galois group $\mathrm{Gal}(\overline{\mathbb{Q}}/\mathbb{Q})$ on $G$. For $\sigma\in \mathrm{Gal}(\overline{\mathbb{Q}}/\mathbb{Q})$ and $f(t)\in \overline{\mathbb{Q}}[t]$ such that $f(t)=\sum_{i=0}^{n}a_{i}t^{i}$ we define
\[f^{\sigma}(t)=\sum_{i=0}^{n}\sigma(a_{i})t^{i}.\]
When $P=0\in E(\overline{\mathbb{Q}}(t))$, we put $\sigma(P)=P$. For $0\neq P=\left(\frac{x_{1}(t)}{x_{2}(t)},\frac{y_{1}(t)}{y_{2}(t)}\right)$ a point in $G$ we define
\[\sigma(P)=\left(\frac{x_{1}^{\sigma}(t)}{x_{2}^{\sigma}(t)},\frac{y_{1}^{\sigma}(t)}{y_{2}^{\sigma}(t)}\right).\]
This action of $\mathrm{Gal}(\overline{\mathbb{Q}}/\mathbb{Q})$ on $G$ preserves the group structure of $G$ since $E$ was defined over $\mathbb{Q}(t)$. This determines a group representation $\rho:\mathrm{Gal}(\overline{\mathbb{Q}}/\mathbb{Q})\rightarrow \mathrm{Aut}(G)$. Group $H$ is equal to the fixed points of $\rho$, namely $H=G^{\rho(\mathrm{Gal}(\overline{\mathbb{Q}}/\mathbb{Q}))}$.

Let us choose a particular automorphism $\tau\in \mathrm{Gal}(\overline{\mathbb{Q}}/\mathbb{Q})$ such that $\tau(\sqrt{-1})=-\sqrt{-1}$ and $\tau(\sqrt{2})=\sqrt{2}$. If $x=a_{1}P_{1}+a_{2}P_{2}+b_{1}T_{1}+b_{2}T_{2}$, $a_{1},a_{2},b_{1},b_{2}\in\mathbb{Z}$ lies in $H$, then $\tau(x)=x$. On the generators of $G$ we get
\[\tau(P_{1})=-P_{1},\tau(P_{2})=P_{2},\tau(T_{1})=T_{1},\tau(T_{2})=-T_{2}.\]
In particular, $\tau(x)=x$ implies $2a_{1}P_{1}+2b_{2}T_{2}=0$, so $a_{1}=0$ and $2b_{2}T_{2}$. Point $T_{2}$ is of order $4$, hence $b_{2}\in 2\mathbb{Z}$. Therefore, group $H$ is generated by $P_{1}$, $T_{1}$ and $2T_{2}$.
\end{proof}

\begin{example}
The rank one result from Corollary \ref{corollary:Generic_rank_1_Q_t} can be improved if we allow $f,g,h\in\overline{\mathbb{Q}}[t]$ but such that the curve (\ref{equation:Weierstrass_equation_f_g_h}) is still defined over $\mathbb{Q}(t)$. Let
\begin{align*}
h_{1}&=(-2) ^{1/4}(4\sqrt{2}-t),\\
h_{2}&=(-2) ^{1/4}(4\sqrt{2}+t).
\end{align*}
By Proposition \ref{proposition:h_1_h_2_generators} we obtain $f=\sqrt{-2} (t^2+2^5)$, $g=-2^4 t$ and $h=\sqrt{-2}(2^5-t^2)$, and the curve
\[E:y^2=x(x + 2(t^2 + 2^5)^2) (x - (-2^4 t)^2 ) \]
is defined over $\mathbb{Q}(t)$.
From Lemma \ref{lemma:rank_2_subgroup_E} it follows that the points
\begin{align*}
Q_{1}=&(-g^2,\sqrt{-2}g^2 h)=(-2^{8} t^2, 2^{9} t^2 (t^2 -2^5)),\\
Q_{2}=&(h^2,fgh)=(-2(2^5 - t^2)^2, -2^{5} t (t^4-2^{10}))
\end{align*}
are linearly independent. Following the same method as the one presented in Corollary \ref{corollary:Generic_rank_1_Q_t}, we prove that $Q_{1}$,$Q_{2}$ and $T_{1}$ and $2T_{2}$ generate the group $E(\mathbb{Q}(t))$. 
\end{example}

\begin{example}
Let $\alpha=1$,$\beta=1$ and $\gamma=2$ and consider the rational parametrization of family (\ref{eq:main_family}) with the polynomials $a=1+2t-t^2$, $b=-1+2t+t^2$ and $c=1+t^2$. The corresponding family (\ref{equation:Weierstrass_equation_f_g_h}) is determined by $f=\sqrt{-1}(1+2t-t^2)$, $g=\sqrt{-1}(-1+2t+t^2)$ and $h=\sqrt{-2}(1+t^2)$. So the curve 
\[E : y^2=x(x+(1+2t-t^2)^2)(x+(-1+2t+t^2)^2)\]
is defined over $\mathbb{Q}(t)$. Group $E(\mathbb{Q}(t))_{\textrm{tors}}$ is generated by 
\begin{align*}
T_{1}=&(g^2,0)=(-(-1+2t+t^2)^2,0),\\
T_{2}=&(fg, \sqrt{-1}f(f - g)g)=(1 - 6 t^2 + t^4,2 (-1 + 7 t^2 - 7 t^4 + t^6)).
\end{align*}
Moreover, the group $E(\mathbb{Q}(t))/E(\mathbb{Q}(t))_{\textrm{tors}}$ is isomorphic to $\mathbb{Z}$, generated by the coset determined by the point
\[Q_{2}=(-g^2,\sqrt{-2}g^2 h)=((-1+2t+t^2)^2,2(1+t^2)(-1+2t+t^2)^2).\]
This family of curves was used in the article \cite{Ulas_Bremner} of Bremner and Ulas.
\end{example}

\begin{example}\label{example:parametrization_rank_3}
For the last example we need the notion of twisting the Weierstrass equation by an automorphism of a field. Let $K$ be a field of characteristic $0$ and assume that
\[E:y^2=x^3+Ax^2+Bx\]
is a Weierstrass model of an elliptic curve such that $A,B\in K$. Let $\sigma:K\rightarrow K$ be an automorphism of field $K$. The curve
\[E^{\sigma}:y^2=x^3+\sigma(A)x^2+\sigma(B)x\]
is a Weierstrass model of another elliptic curve over $K$. The map
\begin{align}
E(K)&\rightarrow E^{\sigma}(K)\nonumber\\
(x,y)&\mapsto (\sigma(x),\sigma(y))\label{al:izomorfizm_nad_K}\\
O&\mapsto O\nonumber
\end{align}
establishes an isomorphism of the Mordell-Weil groups $E(K)$ and $E^{\sigma}(K)$. For this particular example, let $K=\overline{\mathbb{Q}}(t)$ and let $E_3$ (it corresponds to the curve in \cite{Naskrecki_Acta} with the same label) be the curve determined by the equation
\[E_{3}:y^2=x(x-(u_{3}^2-1)^{2})(x-4u_{3}^2),\quad u_{3}=\frac{2t}{5+t^2}.\]
Let $E_{4}$ be another curve, determined by
\[E_{4}:y^2=x(x + 2(u_{4}^2 + 2^5)^2) (x - (-2^4 u_{4})^2 ),\quad u_{4}=\frac{-2^4 t}{-10+t^2}. \]
Let $\sigma$ be the automorphism uniquely determined by the property $\sigma(t)=\frac{1}{\sqrt{-2}}t$. It follows that $E_{3}(K)\cong E_{3}^{\sigma}(K)$. Moreover we have a $K$-isomorphism of elliptic curves
\begin{align*}
\phi:E_{3}^{\sigma}&\rightarrow E_{4},\\
(x,y)&\mapsto (s^2 x,s^3 y),\quad s=-2^5\sqrt{-2}.
\end{align*}
The existence of $\phi$ implies that $E_{3}^{\sigma}(K)\cong E_{4}(K)$. 
We define two triples of polynomials
\[f_{3}=u_{3}^2-1,\quad g_{3}=2u_{3},\quad h_{3}=u_{3}^2+1\]
and
\[f_{4}=\sqrt{-2} (u_{4}^2+2^5),\quad g_{4}=-2^4 u_{4},\quad h_{4}=\sqrt{-2}(2^5-u_{4}^2).\]
For $i\in\{3,4\}$ we define points
\begin{align*}
P_{1,i}=&(-(1+\sqrt{2})g_{i}(g_{i}-h_{i}),\sqrt{-1}(1+\sqrt{2})g_{i}(g_{i}-h_{i})(\sqrt{2}g_{i}-h_{i})),\\
P_{2,i}=&((f_{i}-h_{i})(g_{i}-h_{i}),(f_{i}+g_{i})(f_{i}-h_{i})(g_{i}-h_{i})),\\
T_{1,i}=&(g_{i}^2,0),\\
T_{2,i}=&(f_{i}g_{i}, \sqrt{-1}f_{i}(f_{i} - g_{i})g_{i}).
\end{align*}
Moreover, let us define
\begin{align*}
P_{3,3} &= \left(-f_{3}, \frac{(-5 + t^2) u_{3} (-1 + u_{3}^2)}{5 + t^2}\right)\in E_{3}(K),\\
P_{3,4} &= \left(-\frac{2^6}{\sqrt{-2}} f_{4},\frac{2^9 \left(t^2+10\right) u_{4} \left(2^{5}+u_{4}^2\right)}{10-t^2}\right)\in E_{4}(K).
\end{align*}
The following equalities hold
\begin{align*}
\phi(P_{1,3}) &= -P_{1,4}+T_{1,4}+2T_{2,4},\\
\phi(P_{2,3}) &= -P_{2,4}+2T_{2,4},\\
\phi(T_{1,3}) &= T_{1,4},\\
\phi(T_{2,3}) &= T_{2,4},\\
\phi(P_{3,3}) &= P_{3,4}.
\end{align*}
From \cite[Theorem 1.4]{Naskrecki_Acta} it follows that $E_{3}(K)\cong\mathbb{Z}^{3}\oplus\mathbb{Z}/2\mathbb{Z}\oplus\mathbb{Z}/4\mathbb{Z}$. Hence, the same is true for $E_{4}(K)$. But from \cite[Theorem 1.4]{Naskrecki_Acta} we know that $E_{3}(\mathbb{Q}(t))$ is generated by $P_{2,3},P_{3,3}$, $T_{1,3},2T_{2,3}$, so $E_{3}(\mathbb{Q}(t))\cong \mathbb{Z}^{2}\oplus\mathbb{Z}/2\mathbb{Z}\oplus\mathbb{Z}/2\mathbb{Z}$.

An argument similar to the one in Corollary \ref{corollary:Generic_rank_1_Q_t} shows that the group $E_{4}(\mathbb{Q}(t))$ is generated by $2P_{1,4},2P_{2,4},P_{3,4}$, $T_{1,4},2T_{2,4}$, so
\[E_{4}(\mathbb{Q}(t))\cong\mathbb{Z}^{3}\oplus\mathbb{Z}/2\mathbb{Z}\oplus\mathbb{Z}/2\mathbb{Z}.\]
This explains Remark \ref{remark:rank_3_curve_E_4}.
\end{example}

\subsection{Specialization theorem}
In this final section we will explain how we obtain the lower bounds for the Mordell-Weil ranks described in Theorem \ref{theorem:main_theorem_ranks_over_number_fields} and Corollary \ref{corollary:Best_rank_result}. Our main tool is \cite[Theorem 20.3]{Silverman_arithmetic}.
\begin{proof}[Proof of Theorem \ref{theorem:main_theorem_ranks_over_number_fields}]
By the assumptions of the theorem there exists a triple of polynomials $f,g,h\in F[t]$ such that $\alpha f^2+\beta g^2=\gamma c^2$ and we can assume without loss of generality that $\deg g\leq \deg f$ and $\deg f = 2$. Lemma \ref{lemma:rank_2_subgroup_E} implies that the curve
\[E_{t}:y^2=x(x-\alpha f^2)(x-\beta g^2)\]
treated as an elliptic curve over $\overline{\mathbb{Q}}(t)$ satisfies $\textrm{rank } E_{t}(\overline{\mathbb{Q}}(t)) =2$. Moreover, by the specialization theorem of Silverman \cite[Theorem 20.3]{Silverman_arithmetic} there exists and infinite set of $t_{0}\in F$ such that
\[\textrm{rank }E_{t}(F(t))\leq \textrm{rank }E_{t_{0}}(F).\]
Parameter $t_{0}$ determines a triple $(a,b,c)\in F^{3}$ as follows
\[\frac{a}{c}=\frac{f(t_{0})}{h(t_{0})},\quad \frac{b}{c}=\frac{g(t_{0})}{h(t_{0})}.\]
The specialization homomorphism that defines a map from $E_{t}(F(t))$ to $E_{t_{0}}(F)$ is injective for the parameter $t_{0}$ we chose. The points $Q_{1}$, $Q_{2}$ from the formulation of Lemma \ref{lemma:rank_2_subgroup_E} are linearly independent, hence the specializations of those points are also linearly independent in $E_{t_{0}}(F)$. They have the following form
\begin{align*}
\widetilde{Q}_{1}=&(-\beta b^2,\sqrt{-2}\beta\sqrt{\gamma}b^2c),\\
\widetilde{Q}_{2}=&(\gamma c^2,\sqrt{\alpha}\sqrt{\beta}\sqrt{\gamma}abc).
\end{align*}
Now we apply conditions $(i)$, $(ii)$ and $(iii)$ to show that the lower bound of $\textrm{rank } E_{t_{0}}(F)$ is $1$, $1$ and $2$, respectively.
\end{proof}

\begin{proof}[Proof of Corollary \ref{corollary:Best_rank_result}]
The proof easily follows from the Example \ref{example:parametrization_rank_3}, where we have proved that curve $E_{4}$ has rank $3$ over $\mathbb{Q}(t)$. Silverman's specialization theorem now implies that for any specialization $E_{4,t_{0}}$ of curve $E_{4}$ with parameter $t_{0}$ outside a finite set of rational numbers the rank of $E_{4,t_{0}}(\mathbb{Q})$ is at least $3$. The set in which we can achieve rank at least $3$ is defined in \eqref{equation:specialized_rank_3_set}. In Remark \ref{remark:explicit_rank_3_subgroup} we give explicitly three linearly independent points in $E_{4,t_{0}}(\mathbb{Q})$.
\end{proof}

By specialization it is possible to find curves (\ref{eq:main_family}) over $\mathbb{Q}$ that have at the same time positive rank over $\mathbb{Q}$ and the torsion subgroup over $\mathbb{Q}$ larger than $\mathbb{Z}/2\mathbb{Z}\oplus\mathbb{Z}/4\mathbb{Z}$, which is the maximal torsion subgroup over $\overline{\mathbb{Q}}(t)$ for this family, as described in Corollary \ref{corollary:torsion_structure_f_g_curves_with_assumptions}.
\begin{example}\label{example:large_torsion}
Equation $-a^2-b^2=-52721c^2$ can be parametrized by
\[a=225+128t-225t^2,\, b=-64+450t+64t^2,\, c=1+t^2\]
and the elliptic curve
\[E_{t}:y^2=x(x+a^2)(x+b^2)\]
has the property that $E_{t}(\mathbb{Q}(t))_{\textrm{tors}}\cong \mathbb{Z}/2\mathbb{Z}\oplus\mathbb{Z}/4\mathbb{Z}$. But also:
\begin{itemize}
\item $E_{1}(\mathbb{Q})\cong\mathbb{Z}\oplus\mathbb{Z}/2\mathbb{Z}\oplus\mathbb{Z}/8\mathbb{Z}$
\item $E_{0}(\mathbb{Q})\cong\mathbb{Z}\oplus\mathbb{Z}/2\mathbb{Z}\oplus\mathbb{Z}/8\mathbb{Z}$
\item $E_{-1}(\mathbb{Q})\cong\mathbb{Z}\oplus\mathbb{Z}/2\mathbb{Z}\oplus\mathbb{Z}/8\mathbb{Z}$
\end{itemize}
\end{example}
We leave as an open problem the following question: is it possible to find polynomials $f,g,h$ such that the corresponding curve (\ref{equation:Weierstrass_equation_f_g_h}) will satisfy over $\mathbb{Q}(t)$ the condition that the torsion subgroup over $\mathbb{Q}(t)$ is $\mathbb{Z}/2\mathbb{Z}\oplus\mathbb{Z}/8\mathbb{Z}$ and the rank over $\mathbb{Q}(t)$ will be positive.

\section*{Acknowledgements}
The author would like to thank Wojciech Gajda for his excellent supervision of author's Ph.D. project. He is also indebted to Jerzy Browkin and Jerzy Kaczorowski for their valuable reviews of author's Ph.D. thesis. The author also thanks Remke Kloosterman and Matthias Sch\"{u}tt for their valuable remarks. The author was supported by the Polish National Science Centre research grant 2012/05/N/ST1/02871. This paper is based on the results obtained in the author's Ph.D. thesis \cite{Naskrecki_thesis}.

\bibliography{bibliography}
\bibliographystyle{amsplain}

\end{document}